\documentclass[12pt]{article}

\usepackage{amsfonts,amsthm,amsmath,latexsym,amssymb}

\usepackage[dvips]{graphics}

\usepackage{epsfig}
\usepackage{graphicx}
\usepackage{color}
\usepackage{tikz}

\usepackage{hyperref}

\usepackage{epsf}
\usepackage{psfrag}
\input epsf.tex
\newdimen\epsfxsize
\newdimen\epsfysize

\newcommand{\be}{\begin{equation}}
\newcommand{\ee}{\end{equation}}
\newcommand{\bes}{\begin{equation*}}
\newcommand{\ees}{\end{equation*}}
\newcommand{\s}{\sigma}

\renewcommand{\l}{\lambda}

\newcommand{\g}{\gamma}
\newcommand{\G}{\Gamma}

\renewcommand{\H}{\mathbb H}
\newcommand{\DD}{\mathbb D}
\newcommand{\R}{\mathbb R}

\renewcommand{\O}{\Omega}

\newtheorem{thm}{Theorem}[section]
\newtheorem{prop}[thm]{Proposition}

\newtheorem{defn}[thm]{Definition}
\newtheorem{cor}[thm]{Corollary}

\newtheorem{lemma}[thm]{Lemma}

\makeatletter

\@addtoreset{equation}{section} \makeatother

\def\1{{\bf 1}}

\begin{document}

\title{\bf Loewner Curvature}
\bigskip
\author{{\bf Joan Lind\footnote{Research supported in part by NSF grant DMS-1100714} }~{\bf and}  {\bf Steffen Rohde\footnote{Research supported in part by NSF grant DMS-1068105}}
\\
\\
}

\maketitle

\abstract{The purpose of this paper is to interpret the phase transition in the Loewner theory as an analog of the hyperbolic variant of the Schur theorem about curves of bounded curvature. We define a family of curves that have a certain conformal self-similarity property. They are characterized by a deterministic version of the domain Markov property, and have constant {\it Loewner curvature}. We show that every sufficiently smooth curve in a simply connected plane domain has a best-approximating curve of constant Loewner curvature, establish a geometric comparison principle, and show that curves of Loewner curvature bounded by 8 are simple curves.}

\bigskip

\section{Introduction and Results}\label{intro}

A classical theorem of A. Schur \cite{S} says that, among all curves in the plane of curvature bounded above by $K$,
the circular arc of (constant) curvature K minimizes the distance between the endpoints (assuming the length of the curve is smaller than $\pi/K$). The analog for the hyperbolic plane was established in \cite{FG}: Again, curves of constant (geodesic) curvature are extremal, but an interesting phenomenon occurs. Namely, if the geodesic curvature (which can simply be defined at the point 0 in the hyperbolic unit disc $\DD$ as half of the euclidean curvature, and then at arbitrary points via isometry) is $\leq1$, then the curve is automatically a simple curve that tends to the boundary of $\DD$, whereas the curves of constant curvature $>1$ are circles in $\DD$ in thus do not tend to the boundary.

A  similar ``phase transition" occurs in the Loewner theory: In the framework of the chordal Loewner equation in the upper half plane (see Section \ref{Section:LE}), if a curve $\gamma$ has driving term $\l$ which is H\"older continuous with exponent $1/2$ and norm $||\l||_{1/2}<4$, then $\gamma$ is a simple curve, whereas for norm $\geq4$ the curve can have self-intersections, see \cite{Li} and \cite{RTZ}.

\begin{figure}
\centering
\includegraphics[scale=0.6]{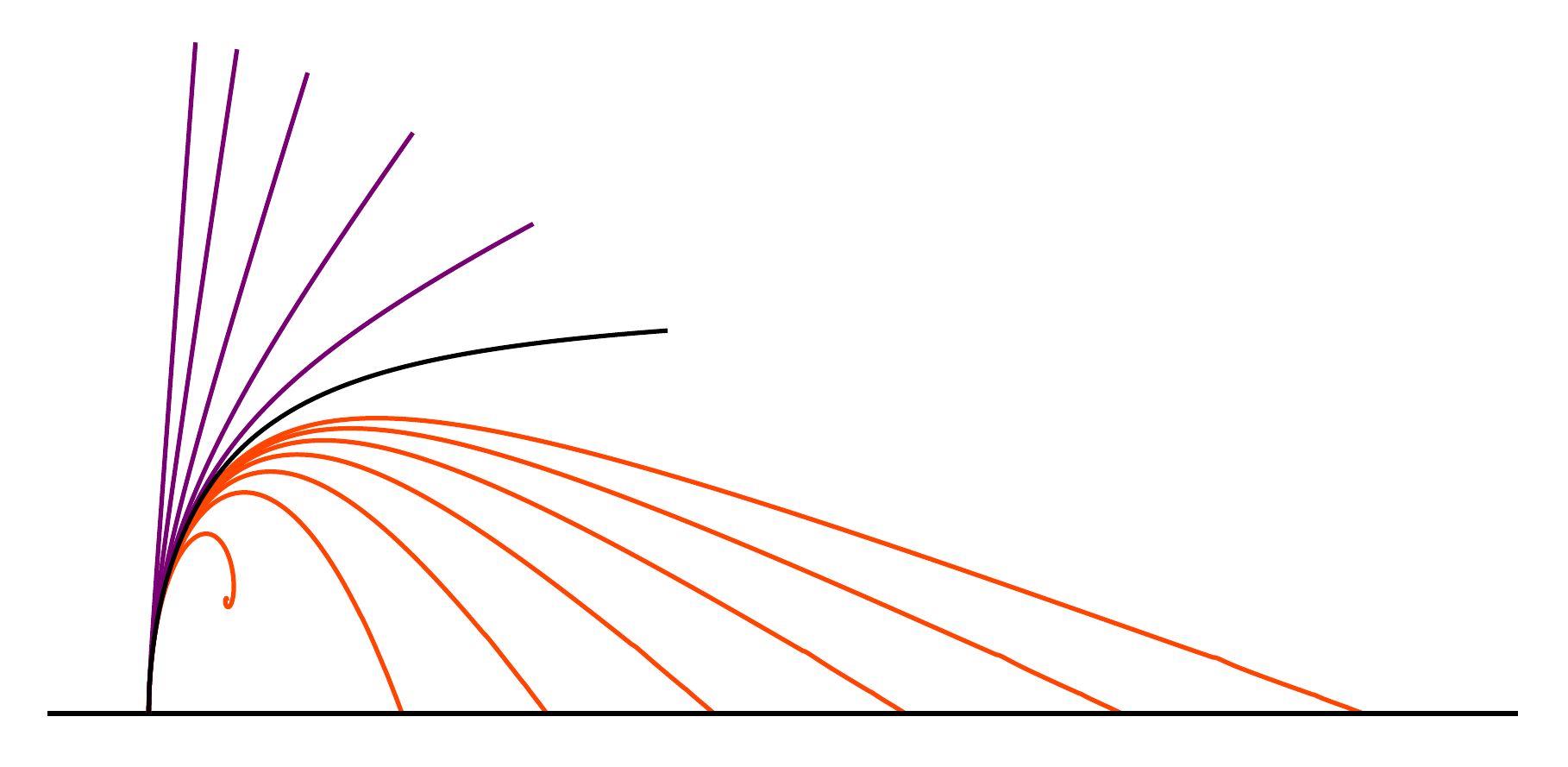}\hfill
\caption
{Some curves of constant Loewner curvature. In our normalization, the purple curves have negative curvature, the black curve has infinite curvature, and the red curves have positive curvature. 
\label{combined}
}\end{figure}

The purpose of this paper is to define a quantity that we call {\it Loewner curvature}, and to establish some of its properties, particularly regarding the above phenomenon. 
Our geometric definition of Loewner curvature will depend on
 a family of curves, introduced in Section 2, that we call {\it self-similar}. 
They are motivated by the Schramm-principle and will be
used as comparison curves. In particular, these are our curves of constant curvature.
In Section \ref{Loewnerframework}, we identify these curves using the Loewner differential equation, and obtain that it is a two-parameter family of curves. A one-dimensional collection of these curves is depicted in Figure 
\ref{combined}. We also compute the first terms of the series expansion of a curve $\gamma$ in terms of the first and second derivative of the Loewner driving term, and show in Proposition \ref{SeriesProp} that 
$$
\g(t) =  2i\sqrt{t} + a \, t - i \frac{a^2}{8}\, t^{3/2} + b\, t^2 + o( t^2 )
$$ 
for $t$ near 0,  where $a = \frac{2}{3} \l'(0)$ and $b = \frac{4}{15} \l''(0) + \frac{1}{135} \l'(0)^3$.
This allows us (Proposition \ref{CompPrin}) to associate with every sufficiently smooth curve a best-fitting comparison curve.  With this foundation, we can make our geometric definition for Loewner curvature  and obtain our Loewner curvature formula
$$
LC_{\g}(t) = \frac{\l'(t)^3}{\l''(t)}
$$
in Section 4. There we also show (Theorem \ref{simpleThm}) that $\g$ is a simple curve if $ LC_{\g}(t) < 8$. The constant $8$ is best possible and corresponds to the constant 4 in the criterion $||\l||_{1/2}<4$ for simple Loewner traces.
In Section 5, we establish a comparison principle (Theorem \ref{below}) by showing that a bound on Loewner curvature
implies that the curve stays to one side of the corresponding curve of constant curvature.

\smallskip

\noindent {\bf Acknowledgement:}  We thank Fredrik Viklund for his comments on the first version of the manuscript.

\section{A family of curves}\label{curves}

In this section, we will introduce a family of curves, which we call ``self-similar curves."
In Section \ref{Lcurvedef}, we will define the Loewner curvature for these curves to be constant.
The Loewner curvature for an arbitrary curve will be defined by comparison to this family. 
In order to give some motivation for our definition of the self-similar curves, 
we will first remind the reader of Schramm's principle  and then formulate a deterministic version.

Let $\O$ be a Jordan domain   with two distinct marked boundary points, $a$ and $b$.  
For each triple $(\O, a, b)$, we assume that there is  
 a family $\G_{\O, a, b}$ of simple curves  in $\overline{\O}$ that begin at $a$
  and    a probability measure $\mu_{\O, a, b}$ on $\G_{\O, a, b}$.  (Note: This could be made slightly more general: we do not need simple curves, but we must require that the curves do not cross over themselves.)
The measures $\mu_{\O, a, b}$ are said to satisfy Schramm's principle if they satisfy the following:
\begin{enumerate}

\item[*] Conformal invariance:  Given a conformal map $\phi$ of the domain $\O$, then
$$\phi(\mu_{\O, a, b}) = \mu_{\phi(\O), \phi(a), \phi(b)}.$$
(The measure $\phi(\mu_{\O, a, b})$  on the family of curves $\phi(\G_{\O, a, b})$ is the push-forward of the measure $\mu_{\O, a, b}$.)

\item[*] Domain Markov property:  Let $\g[0,t]$ be the initial part of a curve in $\G_{\O, a, b}$.  
Let $\mu_{\O, a, b \, |\g[0,t]}$ be the conditional probability measure given $\g[0,t]$.  
Then 
$$ \mu_{\O, a, b\, |\g[0,t]} = \mu_{\O \setminus \g[0,t], \g(t), b}$$
or in other words, given the initial part of the curve $\g[0,t]$, the conditional measure is the same as the measure on curves starting at $\g(t)$ in the slit domain $\O \setminus \g[0,t]$.\\
\end{enumerate}

Given $\g[0,t]$,   define $G_t$ to be the set of all conformal maps
 $\phi_t: \O \setminus \g[0,t] \to \O$
                              with $\phi_t(\g(t)) = a$ and $\phi_t(b) = b$. 
Thus if $\mu_{\O, a, b}$ satisfy conformal invariance and the domain Markov property, 
then  
$$ \phi_t \left( \mu_{\O, a, b \, | \g[0,t]} \right) = \mu_{\O, a, b} $$
for any $\phi_t \in G_t$.
This means that the measure is invariant under  ``mapping down" the initial part of a curve.
Measures that satisfy Schramm's principle must be from the one-parameter family of SLE$_\kappa$ measures.

This leads us to a deterministic Schramm's principle, that is, a deterministic axiom that will characterize an important family of curves.  
In this paper we will be interested in curves in a domain with two marked boundary points.  Let us introduce our notation:
\begin{defn}
Given a Jordan domain $\O$ with distinct boundary points $a$ and $b$ and given $T \in (0, \infty]$,
the notation  $\g:(0, T) \to (\O, a, b)$ means that $\g[0,T)$ is a simple (deterministic) curve  with $\g(0,T) \subset \O$ and $\g(0) =a$.
\end{defn}

We have several comments about the definition:  
First, the point $b$ does not enter into the definition, and in particular, we do not require that $\g(T) =b$.  However,  we think of $b$ as being the location of an observer.  When we introduce our notion of Loewner curvature later, it will be relative to this observer.
Second, we note that when $\g(T)$ is defined,  it is possible that $\g(T) \in \partial \O$ or $\g(T) \in \g(0,T)$.
Next, we consider two curves equivalent if one is a reparametrization of the other.   For the most part, we are not concerned with the particular parametrization of a curve; however when working in the Loewner framework, we will use the typical halfplane-capacity parametrization.
Lastly, for ease of notation, we will often simply write $\g$ for $\g(0,T)$, the image of the open time interval.  

We want to understand the curves  that satisfy the following self-similarity property:

\begin{defn}\label{defss}
A curve $\g:(0, T) \to (\O, a, b)$ is self-similar if $\gamma\in C^3$ and if
for each $t \in (0,T)$, there exists 
a conformal map $ \phi_t  \in G_t$ so that
$$ \phi_t (\g(t,T)) = \g.$$
We write $S(\O, a, b)$ for the family of self-similar curves in the marked domain $(\O, a, b)$.
\end{defn}

In other words, self-similar curves are invariant under ``mapping down" any initial part of the curve, modulo a conformal renormalization.  
In Theorem \ref{CCprop} below, we will determine all self-similar curves.  In particular, we obtain the following: 

\begin{cor}\label{2param}
Given a Jordan domain $\O$ with distinct boundary points $a$ and $b$, there exists a unique two-parameter family of self-similar curves.  
\end{cor}

The assumption on the smoothness of self-similar curves in Definition \ref{defss} is probably unnecessary and only used to establish the regularity needed for our proof of Theorem \ref{CCprop}. It could be replaced by requiring that the $\phi_t$ are continuous in $t.$

\section{Self-similar curves and the Loewner framework}\label{Loewnerframework}

If we know $S(\O_0, a_0, b_0)$ for one fixed triple $(\O_0, a_0, b_0)$, then we can obtain $S(\O, a, b)$ via a conformal map $\psi : \O_0 \to \O$ with $\psi(a_0)=a$ and $\psi(b_0)=b$.  That is,
$$  S(\O, a, b) = \{ \psi \circ \g \, : \, \g \in S(\O_0, a_0, b_0) \}.$$
It is most convenient to work with $\O_0 = \H$, the upper half-plane, with $a_0 = 0$ and $b_0= \infty$, and for the remainder of this section we will work in this setting.  
We take the viewpoint that doing computations with the triple $(\H, 0, \infty)$ is like 
 ``working in coordinates."

Our goal is to use the chordal Loewner equation to understand  the family $S(\H, 0, \infty)$.  We begin with a brief review of the chordal Loewner equation (that the expert can safely skip.)

\subsection{The chordal Loewner equation}\label{Section:LE}

Let $\l: [0,T]  \to \R$ be continuous, and let $z \in \H$.
Then the chordal Loewner differential equation is the following initial value problem:
\begin{align} \label{le} 
\frac{\partial}{\partial t} g_t(z) &= \frac{2}{g_t(z) - \l(t)} , \\
				g_0(z) &= z. \nonumber
\end{align}
We call $\l$ the driving function, since it determines the unique solution $g_t(z)$, which is guaranteed to exist for some time interval.
The only obstacle in solving \eqref{le} is obtaining a zero in the denominator.  We collect these problem points together in the set $K_t$, called the (Loewner) hull.  So
$$K_t := \{ z \in \H \, : \, g_s(z) = \l(s) \text{ for some }  s \in (0,t] \}.$$
If $z \notin K_t$, then $g_t(z)$ is well-defined, and it can be shown that $\H \setminus K_t$ is simply connected and $g_t$ is the unique conformal map from $\H \setminus K_t$ onto $\H$ that satisfies the following normalization at infinity (called the hydrodynamic normalization):
\begin{equation} \label{hydronorm}
g_t(z) = z + \frac{c(t)}{z} + O(\frac{1}{z^2}).
\end{equation}  
Further, $c(t) = 2t$, and this quantity is called the halfplane capacity of $K_t$.  For more details, see \cite{La}.

In the simplest situation, there is a simple curve $\g$ so that  $K_t = \g(0,t]$.
It is also possible that $K_t$ is the complement of the unbounded component of 
$\H \setminus \g(0,t]$ for a non-simple curve $\g$ that touches itself or $\R$.  In both these situations,  the  curve $\g$ is called the trace.
For example, if $\l(t) = 3 - 3\sqrt{1-t}$, then $K_1$ is the simple curve shown on the left of Figure \ref{example1}.
When $\l(t) = 5  -  5 \sqrt{1-t}$, then  $K_1$ is the union of the curve shown in the right of Figure \ref{example1} and the region under the curve.
See \cite{KNK} for a detailed discussion of these examples.
There are other possibilities for $K_t$ (such as a space-filling curve), but these do not concern us in this paper.

\begin{figure}
\centering
\includegraphics[scale=0.5]{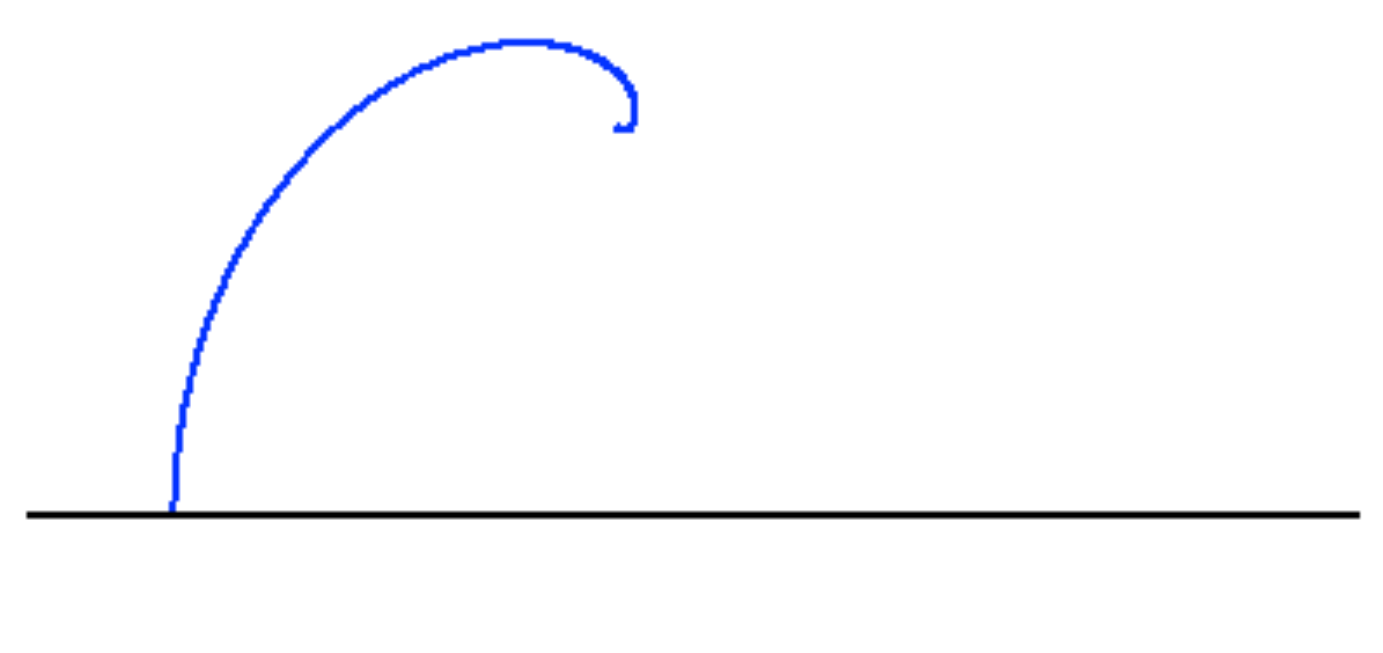} \includegraphics[scale=0.5]{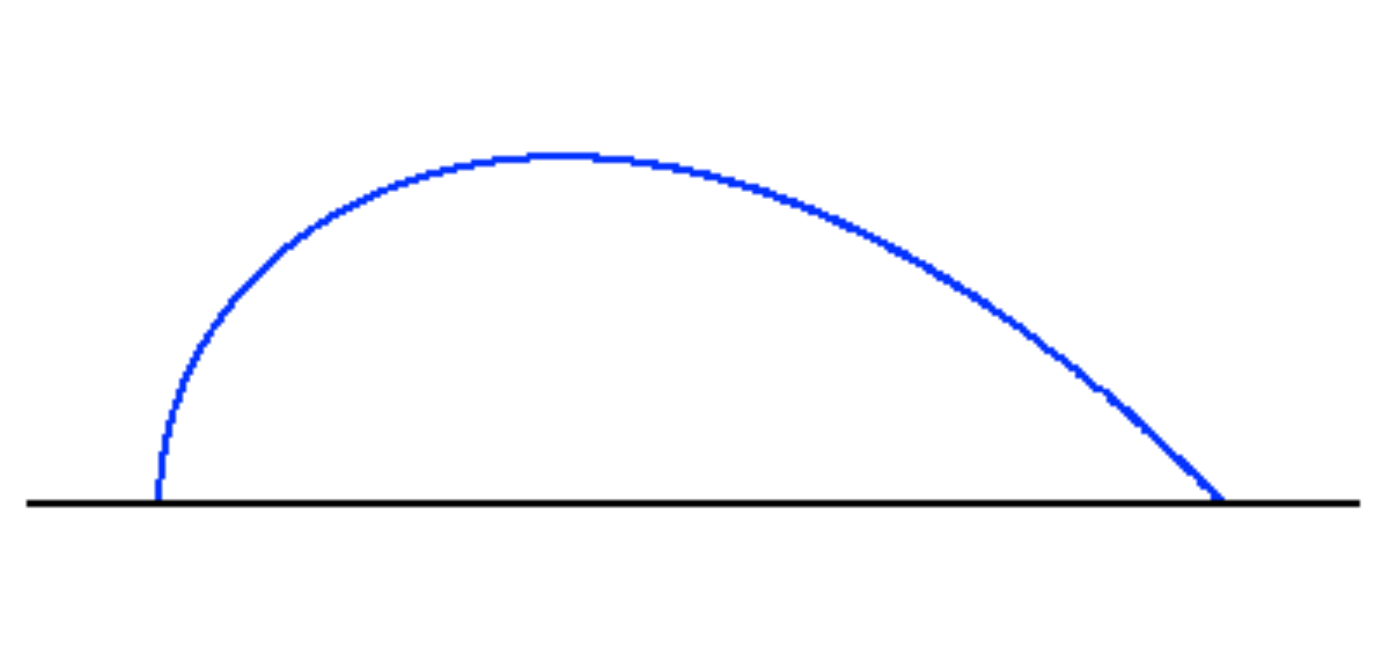}
\caption{The curves generated by driving function $c  -  c \sqrt{1-t}$ when $c=3$ and $c=5$. \label{example1}
}\end{figure}

The Loewner equation provides a correspondence between continuous driving functions and certain families of hulls.  Above, we briefly explained how one can start with a continuous function and use the Loewner equation to obtain a hull.  
On the other hand, if we have an 
appropriate
family of hulls\footnote{For the precise statement, see Section 4.1 in \cite{La}.}, 
we can determine the driving function.  
We will assume for simplicity that the family of hulls is $K_t = \g(0,T]$ for a simple curve $\g$ in $\H$ with $\g(0) \in \R$.  
Let $g_t: \H \setminus \g(0,t] \to \H$ be the conformal map with the hydrodynamic normalization at infinity (the normalization given in \eqref{hydronorm}.)
We can reparameterize $\g$ as necessary so that  in \eqref{hydronorm} $c(t) = 2t$, in which case we say that $\g$ is parametrized by halfplane capacity. 
Then one can show that the conformal maps $g_t$  satisfy \eqref{le} with $\l(t) = g_t(\g(t))$, that is, $\l$ is the conformal image of the tip of the curve $\g$.  

We list some simple but useful properties of the chordal Loewner equation.
Assume that the hulls $K_t$ are generated by the driving term $\l(t)$.  
\begin{enumerate}
\item  {\it Scaling}: For $r>0$, the driving term of
the scaled hulls $r K_{t/r^2}$ is $r \l(t/ r^2 )$.
\item {\it Translation}: For $x \in \mathbb{R}$,
the driving term of shifted hulls $K_t+x$ is $\l(t)+x$.
\item {\it Reflection}:  The driving term of the reflected hulls $R_I(K_t)$ is $-\l(t)$, where
$R_I$ denotes reflection in the imaginary axis.  
\item {\it Concatenation}:
For fixed $\tau$, $\l(\tau+t)$ is the driving function of the mapped hulls $g_{\tau}(K_{\tau+t})$.
\end{enumerate}

We end this section by mentioning
the following theorem, which follows from the work in  \cite{EE} and \cite{M}:
\begin{thm}[Earle, Epstein, Marshall] \label{thmEE}
Assume $\g:(0,T) \to (\H,0,\infty)$ has driving function $\l(t)$.  
If any parametrization of $\g$ is $C^n$, then  the halfplane-capacity parametrization of $\g$ is in $C^{n-1}(0,\tau)$ and $\l \in C^{n-1}(0,\tau)$, where $\tau$ is the halfplane-capacity of $\g$ (and may be infinite.)

\end{thm}

\subsection{Computations in coordinates}

The following theorem implies Corollary \ref{2param}.

\begin{thm} \label{CCprop}
Assume $\g:(0,T) \to (\H,0,\infty)$ has driving function $\l(t)$.
If $\g \in S(\H,0,\infty)$, then $\l$ is completely determined by the two real parameters $\l'(0)$ and $\l''(0)$.
In fact, $\g \in S(\H,0,\infty)$ if and only if 
$\l(t)$ is one of following driving functions:
\begin{equation} \label{drivingfcns} 
0, \;\;\; ct, \;\;\; c\sqrt{\tau} - c\sqrt{\tau-t},  \;\; \text{or } \;\;  c\sqrt{\tau + t} - c\sqrt{\tau}
\end{equation}
 where $c \in \mathbb{R}\setminus \{0\}$ and $\tau >0$.  
\end{thm}

\begin{figure}
\centering
\includegraphics[scale=0.5]{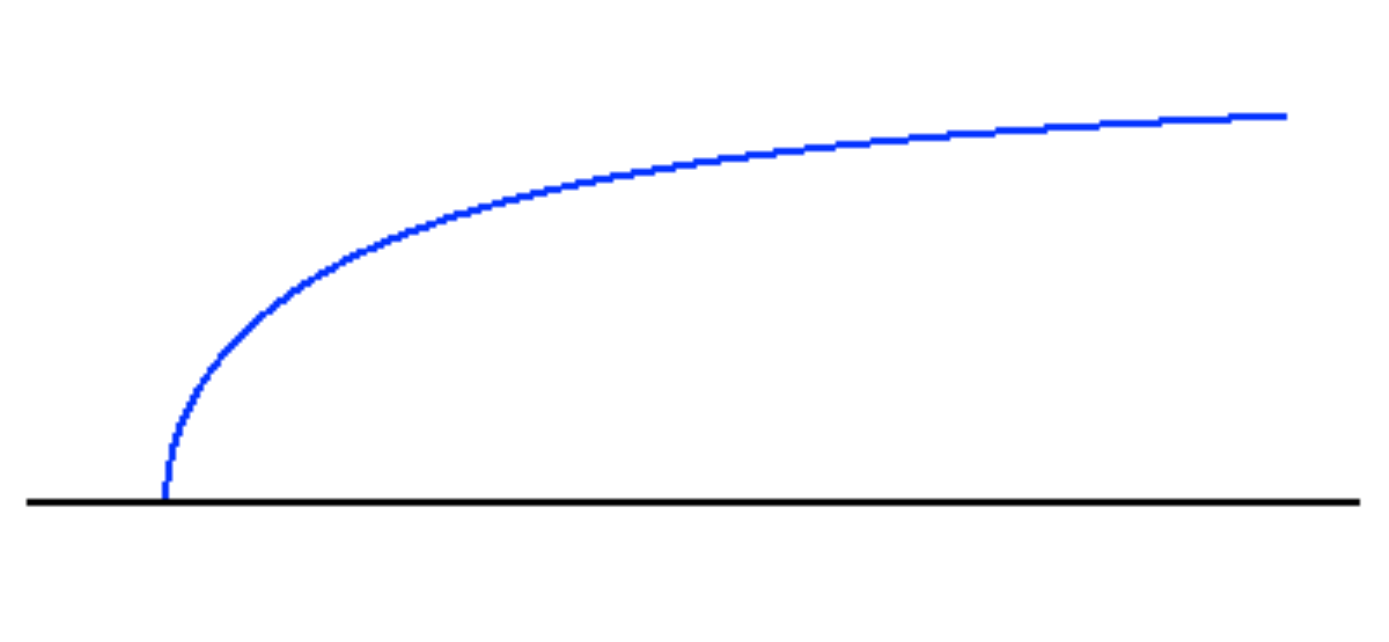} \includegraphics[scale=0.5]{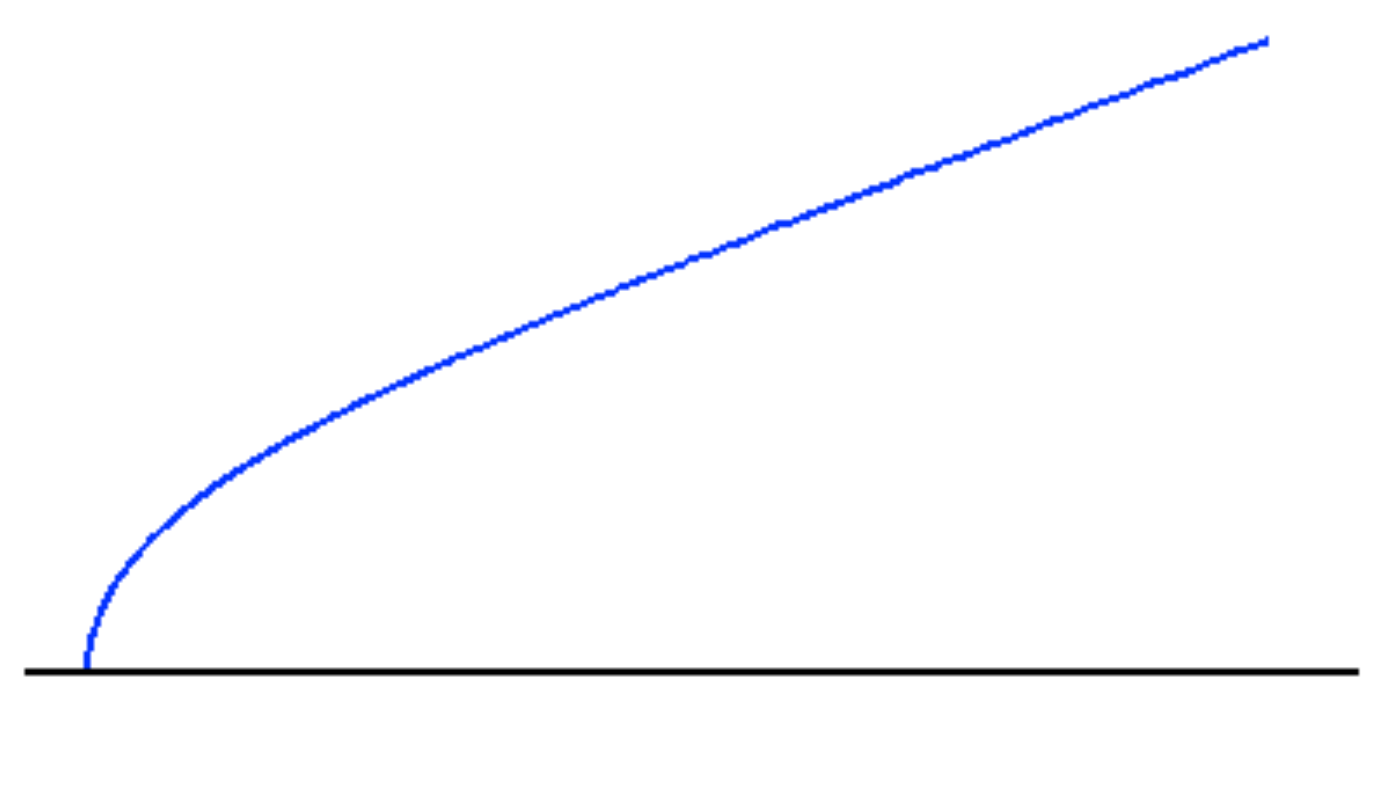}
\caption{The curves generated by driving functions $t$ and $5 \sqrt{1+t} -5$.  \label{example2}
}\end{figure}

Before the proof, we wish to briefly describe the curves generated by the driving functions listed in \eqref{drivingfcns}.   
\begin{enumerate}
\item     If $\l(t)=0$, then $\g$ will be the ray $\{i r \, : \, r >0 \}$.  
 \item    If $\l(t) = ct$, then $\g$ will be a scaled version  of the left curve shown in Figure \ref{example2} (and reflected around the imaginary axis if $c<0$).  This curve goes to infinity, but remains within a bounded distance of the real line.
 \item   If $\l(t) = c\sqrt{\tau} - c\sqrt{\tau-t}$ and $ |c| < 4$, then $\g$ will end with an infinite spiral.  The spiral is so tight as to not be visually discernable, as in the example shown in Figure \ref{example1}.  
        If $\l(t) = c\sqrt{\tau} - c\sqrt{\tau-t}$  and $ |c| \geq 4$, then $\g(\tau) \in \R$, and the angle of intersection depends on $c$. 
        These examples were first described and implicitely computed in \cite{KNK}.
         In both cases, changing $\tau$ scales the picture.
 \item    If $\l(t) = c\sqrt{\tau+t} - c\sqrt{\tau}$, then $\g$ will approach infinity asymptotic to a ray, as shown in the right-hand example in Figure \ref{example2}.   The angle of the ray depends on $c$, and changing $\tau$ scales the picture.\\
\end{enumerate}

\noindent {\bf Remark:} The notion of self-similarity introduced in this paper is a generalization of the notion  introduced in \cite{LMR}, where we defined a  curve $\g$ in $\H$  with finite halfplane capacity  to be self-similar if $g_t(\g(t,T))$ is a translation and dilation of $\g$ for all $t \in (0,T)$.  Proposition 3.1 in \cite{LMR} states that $\g$ is self-similar (under the more restrictive definition) if and only if the driving  function is $\l(t) = k+c\sqrt{\tau -t}$ for some $c, k \in \R$ and $\tau >0$.
\cite{LMR} also contains geometric constructions for these curves, which gives a conceptual and simple way to obtain
the solutions to $\eqref{le}$ with driving function $\l(t) = k+c\sqrt{\tau-t}$.
\\

Now for the proof of Theorem \ref{CCprop}:

\begin{proof}

Assume that $\g \in S(\H,0,\infty)$.  Without loss of generality, we assume $\g$ is parametrized by halfplane capacity.
This means that 
 for each $t \in (0,T)$, there exists a conformal map 
 $ \phi_t: \H \setminus \g[0,t] \to \H$
                               with  $\phi_t(\g(t)) = 0$ and $\phi_t(\infty) = \infty$
 so that
 			$ \phi_t  (\g(t,T)) = \g.$
It follows that $ \displaystyle \phi_t =  r_t \cdot \left(  g_t  - \l(t) \right)$, 
where $r_t >0$ and $g_t$ is the solution to \eqref{le}.
Thus
$$ r_t \cdot \left(  g_t (\g(t,T)) - \l(t) \right) = \g.$$
  By the concatenation, scaling and translation properties, 
\begin{equation} \label{detS}
 r_t \cdot \left(\l \left(t+\frac{s}{r_t^2} \right) - \l(t) \right) = \l(s).
\end{equation}

Theorem \ref{thmEE} implies that $\l \in C^2(0, T)$.
Since the left-hand side of \eqref{detS} is twice differentiable for $s=0$, the same must be true for $\l(s)$, and so $\l \in C^2[0,T)$.
By taking derivatives in \eqref{detS}  with respect to $s$ and setting $s=0$, 
\begin{equation} \label{rteq}
\frac{1}{r_t} \l'(t) = \l'(0) \;\;  \text{   and   } \;\;
\frac{1}{r^3_t} \l''(t)  = \l''(0). 
\end{equation}
If $\l'(0) = 0$, then $\l'(t) = 0$ for all $t$, and hence $\l(t) =0$.
If $\l''(0)=0$, then $\l''(t) = 0$ for all $t$ and $\l(t) = ct$ for $c \in \mathbb{R}$.  
If $\l'(0) \neq 0$ and $\l''(0) \neq 0$, then $\l'(t)$ and $\l''(t)$ are nonzero for all $t$, and \eqref{rteq} gives
\begin{equation} \label{Aeq}
\frac{\l'(t)^3}{\l''(t)} = \frac{\l'(0)^3}{\l''(0)}.
\end{equation}
Setting $A= \frac{\l'(0)^3}{\l''(0)}$ and solving \eqref{Aeq} gives
$$\l'(t) = \pm \left( B -  \frac{2}{A} t \right)^{-1/2} \;\; \text{ and } \;\; 
\l(t) = \pm A \sqrt{B -  \frac{2}{A} t} \, \mp A\sqrt{B} .$$
Thus, there exists $c \neq 0$ and $\tau >0$ so that  $\l(t) = c\sqrt{\tau} - c\sqrt{\tau-t}$ (when $A>0$) or
 $\l(t) = c\sqrt{\tau+t} - c\sqrt{\tau}$ (when $A<0$.)
 
Conversely,  if $\l$ is one of the driving functions $0, ct, \,  c\sqrt{\tau} - c\sqrt{\tau-t}, \, $  or $c\sqrt{\tau + t} - c\sqrt{\tau},$
 then $\l$ will satisfy \eqref{detS}, and this implies that $\g \in S(\H, 0, \infty)$.
\end{proof}

\begin{prop} \label{SeriesProp}
Assume that $\l \in C^2[0,T)$ is the driving function for a  curve $\g$.  Further, we assume that $\l(0) = 0$ and either $\l'(0) \neq 0$ or both $\l'(0)$ and $\l''(0)$ are 0.  Then in the halfplane-capacity parametrization, $\g$ satisfies
\begin{equation} \label{gseries}
\g(t) =  2i\sqrt{t} + a \, t - i \frac{a^2}{8}\, t^{3/2} + b\, t^2 + o( t^2 )
\end{equation}
for $t$ near 0, 
where $a = \frac{2}{3} \l'(0)$ and $b = \frac{4}{15} \l''(0) + \frac{1}{135} \l'(0)^3$.

\end{prop}

\begin{proof}

This proof needs two components:~computations and Theorem 3.3 in \cite{W}.
For the computations, we need to show that the curves generated by driving functions 
$ 0, \;  ct, \;   c\sqrt{\tau} - c\sqrt{\tau-t},$ and $ c\sqrt{\tau + t} - c\sqrt{\tau}$ all satisfy \eqref{gseries}.  
These computations are straightforward.  As done in \cite{KNK}, one can solve \eqref{le} for the driving functions in our list, obtaining an implicit equation for $\g(t)$.  Then we simply expand in a power series to obtain \eqref{gseries}.

For the second step, choose $\l_*(t)$ to be the driving function in this list that satisfies
$\l'(0) = \l_*'(0)$ and $\l''(0) = \l_*''(0)$.  Let $\g_*$ be the  curve generated by $\l_*$.
Then Theorem 3.3 in \cite{W} implies that for $t \in [0, \epsilon]$,
$$|\g(t) - \g_*(t)| \leq C\, \sup_{0\leq t \leq \epsilon} |\l(t)-\l_*(t)|$$
for small enough $\epsilon$.  This in turn implies \eqref{gseries}.

\end{proof}

The previous proposition shows that the two parameters $\l'(0)$ and $\l''(0)$ determine the Loewner trace up to 4th order in $t^{1/2}$ (excluding the case  $\l'(0) = 0$ and $\l''(0) \neq 0$.)  
These two parameters also uniquely determine a self-similar curve in $S(\H, 0, \infty)$ by Theorem \ref{CCprop}. 
As we see in the next proposition, this allows us to match a given curve with a unique ``best-fitting" self-similar curve.
We say $\g_* \in S(\O, a, b)$ is the best-fitting curve to $\g: (0,T) \to (\O, a, b)$ at $a$ 
  if under the conformal transformation $\psi: (\O, a, b) \to (\H, 0, \infty),$ the halfplane-capacity parametrization of
$\psi(\g_*)$ and $\psi(\g)$ match up to 4th order (or satisfy $\l'(0) = \l_*'(0)=0$.)

\begin{prop}
\label{CompPrin}
Let $\g:(0, T) \to (\O, a, b)$ be $C^3$.
At each point $\g(t)$, the curve $\g(t, T)$ has
 a unique best-fitting curve $\g_* \in S(\O \setminus \g[0,t], \g(t), b)$.
\end{prop}

\begin{figure}[h]
\centering
\includegraphics[scale=0.5]{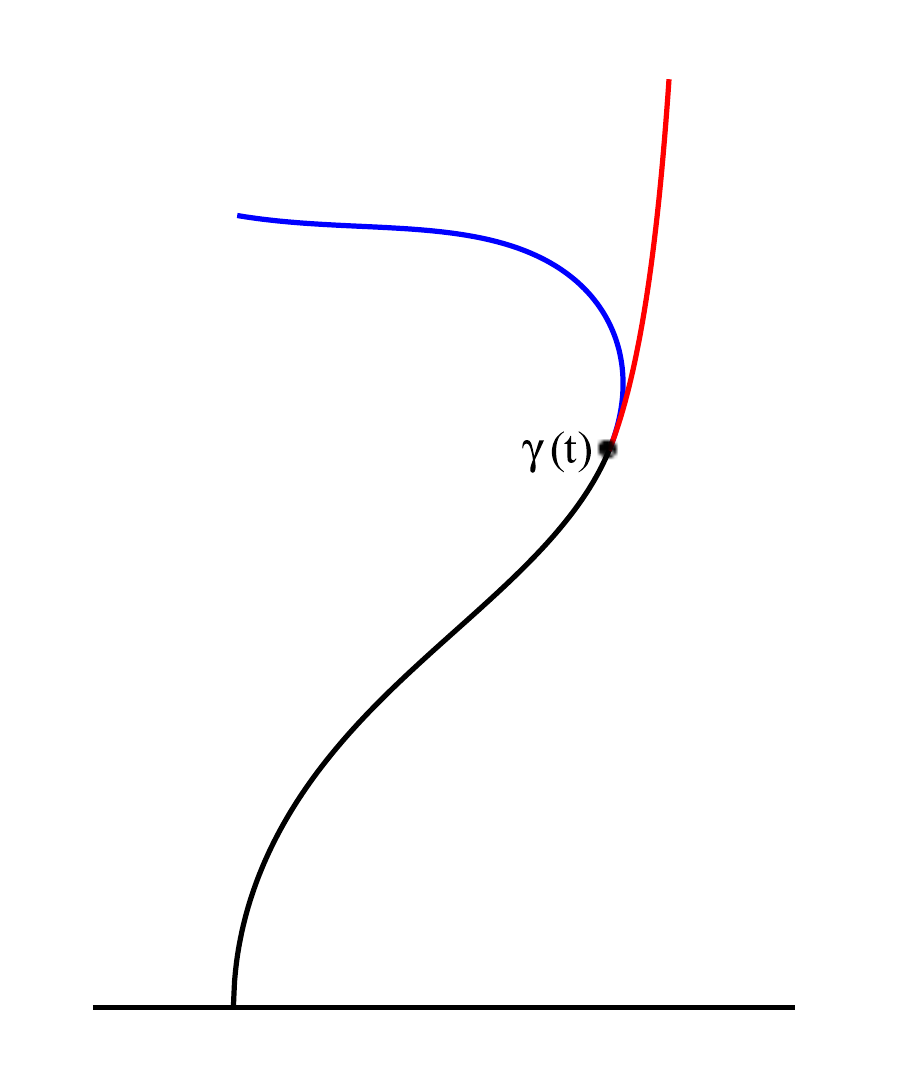}
\caption{Illustration of Proposition \ref{CompPrin} : $\g(0,t)$ shown in black, $\g(t,T)$  in blue and $\g_*$  in red.  }
\label{LCdefpic}
\end{figure}

\begin{proof}
By conformal transformation, it suffices to consider the situation when $\O = \H$, $a=0$ and $b= \infty$.   By Theorem \ref{thmEE}, since $\g$ is $C^3$, we can deduce that both $\l$ and the halfplane-capacity parametrization of $\g$  are in $C^2(0,\tau)$ (where  $\tau$ is the halfplane capacity of $\g$ and may be infinite.)  For the rest of the proof, we assume $\g$ is parametrized by halfplane capacity and $T=\tau$.

For $t \in (0,T)$, $\l'(t)$ and $\l''(t)$ are defined.  If $\l'(t) = 0$, we set $\g_*(t) = 2i\sqrt{t}$.  Otherwise, we choose the unique $\g_* \in S(\H, 0, \infty)$ satisfying $\l'(t) = \l_*'(0)$ and $\l''(t) = \l_*''(0)$.  
Then by Proposition \ref{SeriesProp} and Theorem \ref{CCprop}, $\g_*$ is the unique best-fitting curve to $g_t(\g(t, T)) -\l(t)$ at the origin.
Thus $g_t^{-1} (\g_*) \in S(\H \setminus \g[0,t], \g(t), \infty)$ is the unique best-fitting curve to $\g(t,T)$ at the point $\g(t)$.

\end{proof}

\section{Loewner curvature}\label{Lcurvedef}

\subsection{The definition of Loewner curvature}

We are now ready to give the definition of Loewner curvature, notated $LC_{\g}(t)$.  We begin with assigning constant curvature for $\g \in S(\H, 0, \infty)$ as follows:
\begin{defn} \label{SSLCdef}
Let $\g \in S(\H, 0, \infty)$. Then $LC_{\g}$, the Loewner curvature of $\g$, is defined to be the following constant:
\begin{enumerate}
\item   If $\g$ is generated by $ 0$, then $LC_{\g} \equiv 0$.
\item   If $\g$ is generated by $ ct$, then $LC_{\g} \equiv \infty$.
\item   If $\g$ is generated by $ c\sqrt{\tau} - c\sqrt{\tau-t}$, then  $LC_{\g} \equiv c^2/2$.
\item   If $\g$ is generated by $ c\sqrt{\tau + t} - c\sqrt{\tau}$, then $LC_{\g} \equiv -c^2/2$.
\end{enumerate}
\end{defn}
Note that for $\g \in S(\H, 0, \infty)$,  the Loewner curvature is scaling invariant by definition.
For the self-similar curves $\g \in S(\O, a, b)$, $LC_{\g}(t)$  is defined to satisfy  conformal invariance. 
For a $C^3$ curve $\g:(0,T) \to (\O, a, b)$, we  
define $LC_{\g}(t)$  by comparison to the curves of constant curvature 
as follows:
\begin{defn}\label{LCdef}
Let  $\g:(0,T) \to (\O, a, b)$ be $C^3$.
The Loewner curvature of $\g$ at the point $\g(t)$, notated $LC_{\g}(t)$, is defined to be 
$ LC_{\g_*}$, where
$\g_* \in S(\O \setminus \g(0,t], \g(t), b)$ is the unique best-fitting curve to $\g(t,T)$ at $\g(t)$. 
\end{defn}
\noindent Proposition \ref{CompPrin} (illustrated in Figure \ref{LCdefpic}) guarantees that  $\g_*$  exists.

Although our approach to defining Loewner curvature via comparison curves is natural, it is difficult to use this definition to compute $LC_{\g}(t)$.  
The following lemma shows that computations are straightforward  in the Loewner framework. 
 (The nice formula that appears in this lemma also explains the particular definition for $LC_{\g}$ 
 in cases 3 and 4 of Definition \ref{SSLCdef}.)

\begin{prop} \label{curvProp}
Let $\g:(0,T) \to (\H, 0, \infty)$ be a $C^3$ curve that is parametrized by halfplane capacity, and let $\l$ be the corresponding driving function.  
Then for $t \in (0,T)$,  
\begin{equation} \label{curv}
LC_{\g}(t) = \frac{\l'(t)^3}{\l''(t)}.
\end{equation}
\end{prop}

Since $LC_\g(t) = 0$ precisely when $\l'(t) = 0$, if the right-hand side of \eqref{curv} is in the indeterminate form $``\,\frac{0}{0}\, "$, we declare it to equal 0.

\begin{proof}
By Definition \ref{SSLCdef},  equation \eqref{curv} is true for $\g \in S(\H, 0, \infty)$.
For $t \in (0, T)$, let $\g_*  \in S(\H \setminus \g(0,t], \g(t), \infty)$ be the unique best-fitting curve to $\g(t,T)$ at $\g(t)$.  Suppose that $\l_*$ generates $g_t (\g_*) - \l(t)$, which is the unique best-fitting curve in $  S(\H, 0, \infty)$ to $g_t(\g(t, T))-\l(t)$ at 0. 
Then
$$LC_{\g}(t) = LC_{\g_*}= \frac{\l_*'(0)^3}{\l_*''(0)}= \frac{\l'(t)^3}{\l''(t)}.$$
The last equality is due to Proposition \ref{SeriesProp}. 

\end{proof}

The next theorem says that if the curvature is small, then $\g$ cannot ``curve" enough to hit itself or the boundary (except at the marked point $b$.)

\begin{thm}  \label{simpleThm}
Let $\g:(0,T) \to (\O,a,b)$ be $C^3$.  
If $ LC_{\g}(t) < 8$  for all $t \in (0,T)$,  
then $\g(0,T]$ is a simple curve in $\O \cup \{b\}$.
\end{thm}

The constant 8 is best possible, as the curve in case 3 of Definition  \ref{SSLCdef} with $c=4$ shows.
The following lemma is key component of the proof.

\begin{lemma} \label{lipLemma}
Let $\g:(0,T) \to (\H, 0, \infty)$ be $C^3$ with driving function $\l$.  
  If $0< LC_{\g}(t) \leq A$, then $\lambda$ is Lip$(1/2)$ with norm bounded above by $\sqrt{2A}$.

\end{lemma}

\begin{proof}

Since $\g$ is $C^3$, Theorem \ref{thmEE} implies that  $\l \in C^2(0,T)$.  
By Proposition \ref{curvProp}, 
 $0 < \lambda'(t)^3 / \lambda''(t) \leq A$ for all $t \in (0,T)$.   
 Without loss of generality, we assume that $T$ is finite, and
we will show that $|\lambda(t)-\lambda(s)| \leq \sqrt{2A} \sqrt{|t-s|}$ for $t,s \in [0,T]$.

Set $\sigma=\lambda'$, and  assume  that $\sigma > 0$  (if not, replace $\sigma$ with $-\sigma$.) 
  Then
\begin{equation}\label{sigmaequ}
   0< \sigma^3 \leq A\sigma'.
\end{equation}
The solution to  $\sigma^3=A\sigma'$ is
$$\sigma_{A,B}(t)= \sqrt{A/2}(B-t)^{-1/2}$$
for some $B>0$. 
Assume for the moment that $\sigma \leq \sigma_{A,T}$, and
let $0 \leq s < t \leq T$.  Then
$$\lambda(t)-\lambda(s) = \int_s^t \sigma(u) \, du \leq \int_s^t \sigma_{A,T}(u) \, du \leq \sqrt{2A}\sqrt{t-s}.$$

It remains to show $\sigma(t) \leq \sigma_{A,T}(t)$ for all $t \in (0,T)$.  
Suppose to the contrary that there is a time $s_0<T$ so that $\sigma(s_0) > \sigma_{A,T}(s_0)$. 
This will imply, as we show in the remainder of the proof,  that there is a time $\tau < T$ with $\sigma(\tau) = \infty$, contradicting the fact that $\sigma \in C^1(0,T)$.

For simplicity, assume $s_0 = 0$ (by shifting time if necessary), and set 
\begin{equation} \label{alphaequ}
	\alpha = \sigma(0)/\sigma_{A,T}(0) >1.
\end{equation}
Let $N$ be large enough so that 
$$N^2 \sum_{n=N}^{\infty} \frac{1}{n^3} < \frac{\alpha^2}{2}.$$
This is possible since $\alpha>1$ and the limit as $N \to \infty$ of the left hand side is $1/2$.
Recursively define an increasing
sequence $\{s_n\}$ by
$$s_{n+1} = s_{n} + \frac{2T}{\alpha^2} \frac{N^2}{(N+n)^3}.$$
Finally set $\displaystyle \tau = \lim_{n \to \infty} s_n$, which satisfies
$$\tau = \sum_{n=0}^{\infty} \left( s_{n+1} - s_{n} \right) = \frac{2T}{\alpha^2} N^2 \sum_{n=0}^\infty \frac{1}{(N+n)^3} < T.$$

We will show by induction that
\begin{equation}\label{inductgoal}
 \sigma(s_n) \geq \alpha \sqrt{\frac{A}{2T}} \, \frac{N+n}{N}.
\end{equation}
Rewriting \eqref{alphaequ} gives $\sigma(0) = \alpha \sqrt{A/2T\,}$, which is the base case.
Assume that \eqref{inductgoal} holds for some fixed $n$. 
By \eqref{sigmaequ}, $\sigma$ is increasing and $\sigma' \geq \sigma^3/A$.  
Thus,
\begin{align*}
\sigma(s_{n+1})
 &= \sigma(s_n) + \int_{s_n}^{s_{n+1}} \sigma'(t) \,dt \\
		&\geq \sigma(s_n) + \frac{1}{A} \sigma(s_n)^3 \,\left( s_{n+1} - s_n \right)  \\
		&\geq  \alpha \sqrt{\frac{A}{2T}} \,\frac{N+n}{N}\left( 1 + \frac{1}{A} 
			\left( \alpha \sqrt{\frac{A}{2T}} \, \frac{N+n}{N}\right)^2 \,
			 \frac{2T}{\alpha^2} \frac{N^2}{(N+n)^3}  \right) \\
		&=  \alpha \sqrt{\frac{A}{2T}}  \frac{N+n+1}{N}.
\end{align*}
Thus \eqref{inductgoal} is true for all $n$, showing that $\sigma(\tau) = \infty$.
\end{proof}

Now the proof of Theorem \ref{simpleThm}:

\begin{proof}

It suffices to prove Theorem \ref{simpleThm} for the triple $(\H, 0, \infty)$.  Let $\l$ be the driving function for $\g$, in which case Proposition \ref{curvProp} implies that $LC_{\g}(t) = \l'(t)^3/\l''(t)$.  
Thus  $\l'(t)^3/\l''(t) \leq 8$ for all $t \in (0,T)$.

We split this proof into two cases.  For the first case, suppose that there exists a time $t_0$ so that $\l'(t_0)^3/\l''(t_0)>0$.  
We claim that $\l'(t)^3/\l''(t) \in (0, \infty)$ for all $t \in (t_0, T)$.  
If this claim is not true, then there must be a first time $t_1 \in (t_0, T)$ so that $\l'(t_1)^3/\l''(t_1)\notin (0, \infty)$.
This implies that either $\l'(t_1) = 0$ or $\l''(t_1)=0$.
However, for $t \in [t_0, t_1)$,
$\,\l'(t)$ and $\l''(t)$ must have the same sign, implying that $0 < |\l'(t_0)|^3 \leq |\l'(t)|^3 \leq 8 |\l''(t)|$.  
Hence, $\l'(t)$ and $\l''(t)$ are bounded away from zero on $[t_0, t_1]$, proving the claim.
Now in this case, Lemma \ref{lipLemma} implies that $\l$ is Lip$(1/2)$ on $[t_0, T]$ with norm strictly less than 4.  
Thus by Theorem 2 in \cite{Li},  $g_{t_0}(\g(t_0,T])$ is a simple curve in $\H \cup \{ \infty \}$.
  Therefore $\g(0,T]$ is simple in $\H \cup \{ \infty \}$.

For the next case, suppose that $\l'(t)^3/\l''(t) \in [-\infty, 0]$ for all $t \in (0,T)$. This means that $\l'$ and $\l''$ always have the opposite sign (or could be 0).  
Thus, for $\epsilon >0$, $\l'$ is bounded on $[\epsilon,T)$, which implies that $\l$ is locally Lip$(1/2)$ with small norm. 
In particular, we can find an interval $(t_0,T)$ so that $\l$ is Lip$(1/2)$ on $(t_0, T)$ with norm strictly less than 4.  We finish the proof as in the first case.

\end{proof}

 For $\g:(0,T) \to (\H, 0, \infty)$,  Loewner curvature is invariant under scaling.  
 This follows from Definitions \ref{SSLCdef} and \ref{LCdef}  and can also be verified computationally by Proposition \ref{curvProp}.
 To distinguish between different curves with the same Loewner curvature, we must choose a ``scale" 
 which we define as follows:

\begin{defn}
Assume  $\g:(0,T) \to (\H,0, \infty)$ has driving function $\l \in C^2[0,T)$ with $\l'(0)\neq 0$.
Then the scale of $\g$ at 0 is $a= \frac{2}{3} \l'(0)$, which is 
 the coefficient  of the linear term in   \eqref{gseries}.
\end{defn}

The scale of $\g$ at 0 is simply a multiple of the Euclidean curvature of $\g$ at 0.   To see this, let $\g_r$ be the upper half-circle with radius $r$, which is driven by $3r - 3\sqrt{2} \sqrt{r^2/2 - t}$. 
Thus,
Proposition \ref{SeriesProp} implies that 
$$\g_r(t) = 2i\sqrt{t} + \frac{2}{r} t +O(t^{3/2}).$$ 
Comparing this to \eqref{gseries}
shows that the half-circle that best matches the curve $\g$ at its base has radius $r=2/a = 3/\l'(0)$.
In other words, the scale of $\g$ at 0 is twice the Euclidean curvature of $\g$ at its base.
\\

\subsection{Loewner curvature is neither local, nor reversible}
\begin{figure}[h]
\centering

\begin{tikzpicture}
\draw[blue] (0,0) to [out=90,in=0] (-1,0.5) to [out=180,in=270] (-3,2)
to [in=180,out=90] (-1,4) to [out=0,in=180] (-0.5,1) to [in=270,out=0] (0,1.5) 
to [in=270,out=90] (0,6);
\draw[thick] (-5,0) -- (5,0);
\draw[fill] (0,0) circle [radius=0.025];
\node[below] at (0,0) {$0$};
\draw[fill] (0,2) circle [radius=0.025];
\node[right] at (0,2) {$2i$};
\end{tikzpicture}

\caption{A curve for which Loewner curvature is not local and not reversible. }
\label{NotLocalpic}
\end{figure}
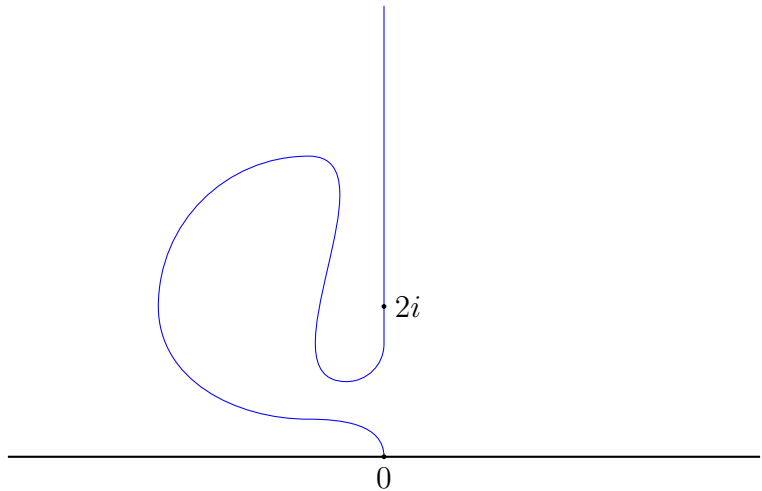

Although Loewner curvature shares some similarities with other notions of curvature, there are  significant differences.  For instance, Loewner curvature is not reversible (but depends on the orientation of the curve), and it is not local (but depends on the ``past" of the curve).  
To see why this must be true, let's assume for the moment that Loewner curvature is a purely local concept, 
and consider the smooth curve $\gamma$ shown in Figure \ref{NotLocalpic}.  Let $t_0$ be the time that $\g(t_{0}) = 2i$.
If Loewner curvature were purely local, then $LC_{\g}(t) =0$ for $t > t_{0}$, since locally near $\g(t)$ the curve $\g$ looks like the vertical slit, which has constant Loewner curvature 0.  This means that $g_{t_0}(\g)$ must be a vertical ray, and therefore $\gamma(t_0,\infty)$ would be a hyperbolic geodesic in $\H\setminus\g[0,t_0]$, which it is not.
This example also shows that Loewner curvature cannot be reversible: in contrast to the forward-direction, if we traverse $\g$ from $\infty$ to 0, then the Loewner curvature will be  0 prior to  reaching $2i$.  
\\

\subsection{Existence of a curve with given Loewner curvature}

It is natural to ask the following question:\\

\noindent {\bf Question:} Given a continuous function $l:[0,T] \to \R \cup \{ \infty \}$,
              does there exist a curve $\g:(0,T) \to (\H,0, \infty)$ with $LC_{\g}(t) = l(t) $?\\

We will need to revise this question before we can answer it in Theorem \ref{specifiedcurvature} below. 
To understand the needed revision, we consider two examples.
First suppose $l$ is a nonzero constant function. 
 Then  there are an infinite number of self-similar curves  with Loewner curvature equal to $l$.  Thus we should have the freedom to specify another parameter, for instance, the scale (or equivalently, the Euclidian curvature) of the curve at 0. 
  However, if we are in the case that the curve is driven by $\l(t) = c\sqrt{\tau} - c \sqrt{\tau -t}$, 
then   specifying the scale determines the halfplane capacity $\tau$ of the curve.  
There is a difficulty if $\tau$ is smaller than the desired value of $T$.  
Since $\l'(t) \to \infty$ as $t \to \tau$, it is impossible to continue $\l$ past $\tau$ so that $\l$ remains $C^2$.

For our second example, suppose $l(t) = (1-t)^2$ for $t \in [0,2]$, and let $a\neq0$ be the scale of the curve.  Now we wish to solve $\l'(t)^3/\l''(t) = (1-t)^2$ with $\l'(0) = 3a/2$.  
 Separation of variables leads to
 $$\l'(t)^{-2} = \frac{4}{9\,a^2} - \frac{2t}{1-t}.$$
Regardless of the value of $a$, $|\l'(t)| \to \infty$ before time 1.
 These examples show that the best we can hope for is the existence of a curve on a small time interval.

\begin{thm}\label{specifiedcurvature}
Let  $l:[0,T] \to \R\cup \{\infty\}$ be a continuous function  with $l(0) \neq 0$, and let $a \in \mathbb{R} \setminus \{ 0 \}$.
Then
 there exists $\tau>0$ and a curve $\g:(0,\tau) \to (\H,0, \infty)$ so that
$LC_{\g}(t) = l(t) $ for $t\in(0,\tau)$ and $\g$ has scale $a$ at 0.
\end{thm}

\begin{proof}
Since $l(0) \neq 0$, there is some interval $[0,\tau]$ so that $l(0)$ is bounded away from zero.  
Then we can solve solve 
$$\s(t) = l(t) \s'(t), \;\;\;\;\; \s(0)=3a/2$$
by separation of variables to obtain
$$\s(t)^{-2} = \frac{4}{9 \, a^2} - 2 \int_0^t \frac{1}{l(s)} \, ds.$$
By taking a smaller  value of $\tau$ if needed, we can ensure that the right-hand side is positive and bounded away from zero for all $t \in [0, \tau]$.  Therefore, we can integrate $\s(t)$ to obtain a function $\l(t)$ defined on $[0,\tau]$.
Now let $\g$ be the curve generated by $\l$ via the Loewner equation.
Proposition \ref{curvProp} implies that $LC_{\g}(t) = l(t)$, 
 and Proposition \ref{SeriesProp} implies that the scale of $\g$ at 0 is a.

\end{proof}

\section{A comparison principle}\label{curvaturebounds}

In Theorem \ref{simpleThm}, a bound on the Loewner curvature yielded global information about a curve.  The next theorem is in the same spirit.
Let us first set the stage, by 
 introducing the normalized curves of constant curvature, $\gamma^c$ and $\Gamma^c$.
For $c>0$, set 
$$\l_c(t) = c^2-c\sqrt{c^2-t} \;\; \text{ and } \;\; \Lambda_c(t)=c\sqrt{c^2+t}-c^2. $$ 
Let $\gamma^c$ be the curve  generated by $\l_c$, and 
similarly, let $\Gamma^c$ denote the curve generated by $\Lambda_c$.
Each of these curves has the same scale, since $\l_c'(0) = 1/2 = \Lambda_c'(0)$.
See Figures \ref{ccc} and \ref{cccn}.

\begin{figure}
\centering
\includegraphics[scale=0.8]{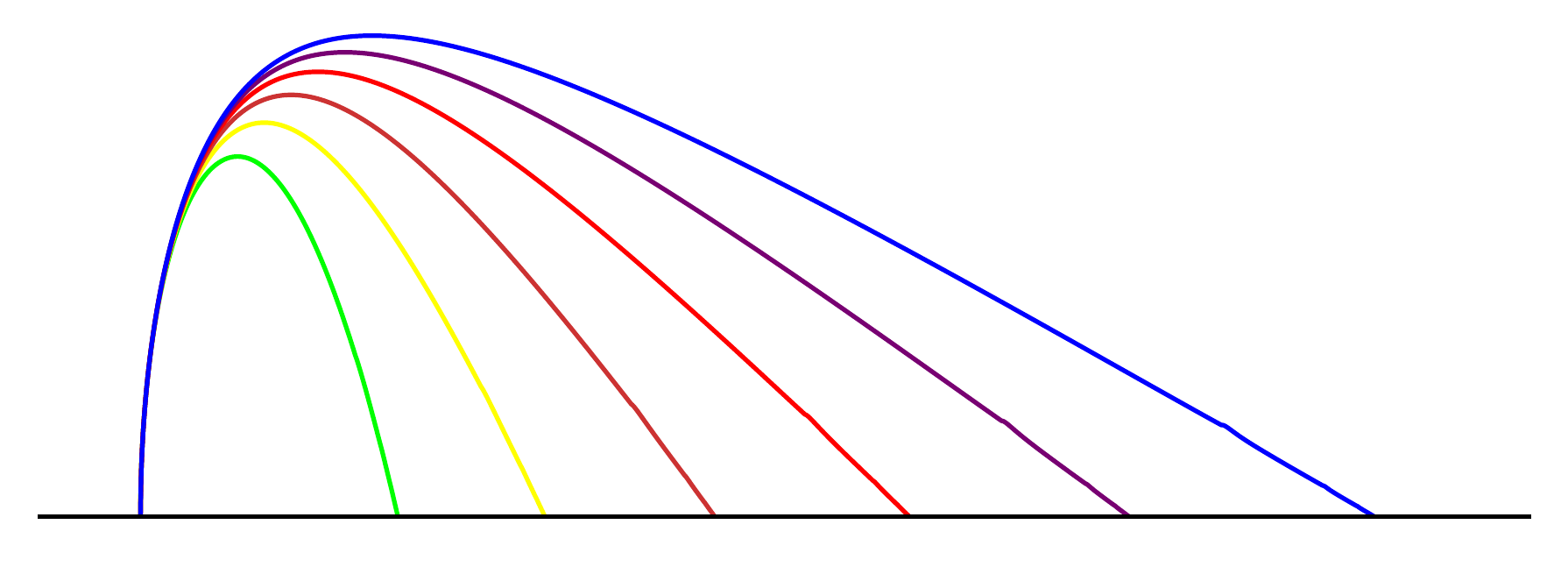}\hfill
\caption{Curves $\g^c$ generated by $c^2-c\sqrt{c^2-t}$ for $c=5,6,7,8,9,10$.   \label{ccc}
}\end{figure}

We list some additional terminology and notation that is needed:
\begin{enumerate}
\item[*] For $c \geq 4$, define $\tau_c = c^2$, and for $c < 4$, define $\tau_c$  to be the first time that the tangent vector $\g_c'(t) $ points downward.
We will often use $\g$  for $\g[0,T]$ and $\g^c$ for $\g^c[0,\tau_c]$. 
\item[*] The phrase ``the base of $\g_1$ is to the right of the base of $\g_2$" means that for $\g_1$ and $\g_2$ parametrized by height $h$, 
$\text{Re} \left( \g_1(h) \right) \geq \text{Re} \left( \g_2(h) \right)$ for small $h$. 
\item[*] The phrase ``$\g_1$ is below $\g_2$" means that the base of $\g_1$ is to the right of the base of $\g_2$ and the curves $\g_1$ and $\g_2$ never cross (although we allow them to touch.)
\end{enumerate}

\begin{thm}\label{below}
Assume $\g$ is generated by $\lambda \in C^2[0,T)$ with  $\l(0)=0$ and  $\l'(0)=1/2$.  Let $c>0$.
 \begin{enumerate}
    \item  If  $0< LC_{\g}(t) \leq c^2/2,$
   then  $\g[0,T)$  is below $\g^c[0,\tau_c]$.
     \item  If  $c^2/2 \leq LC_{\g}(t) < \infty,$  
  then $\g^c[0,\tau_c]$ is below the curve $\g[0,T)$.
   \item  If $-\infty < LC_{\g}(t)  \leq -c^2/2, $ 
     then  $\g[0,T)$  is below $\Gamma^c[0,\infty)$.
  \item If $-c^2/2 \leq LC_{\g}(t) <0, $ 
     then $\Gamma^c[0,\infty)$ is below  $\g[0,T)$.

\end{enumerate}
\end{thm}

The proof of each part of this theorem has two steps.  
First we must analyze the base of $\g$ using the power series expansion given in Proposition \ref{SeriesProp}.
For the second step, we analyze the curve away from its base by comparing the flow under $\l$ and the flow under $\l_c$ (or $\Lambda_c$).  This involves changing time for one of the flows to allow easier comparison.

For $g_t$ and $\g$  generated by $\l$, 
 set
 $$\g_t = g_t(\g)-\l(t),$$ 
that is, we ``map down" $\g[t,T]$ by the function $g_t$ and then shift so that the curve is rooted at the origin.  
We also have the notation $g^c_t, \g^c,$ and $ \g^c_t$ for the corresponding functions and curves associated with the driving function $\l_c$.   
We wish to compare $\g_t$ and $\g^c_t$.  However, we want to compare curves that are the same scale. 
This leads us to introduce the following time change:
for $t \in[0,T]$, set $$s = s(t)= (\l_c')^{-1}(\l'(t)).$$
In particular, we will have that $\l_c'(s) = \l'(t)$.

\begin{figure}
\centering
\includegraphics[scale=0.8]{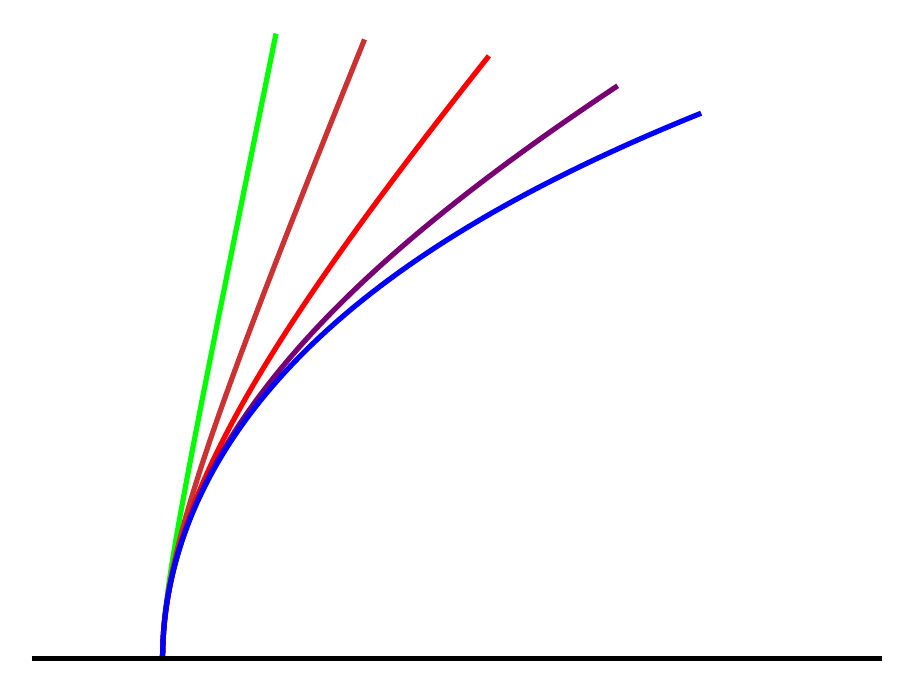}\hfill
\caption{Curves $\Gamma^c$ generated by $c\sqrt{c^2+t}-c^2$ for $c=1/2, 1, 2, 4, 8$.   \label{cccn}
}\end{figure}

\begin{lemma}\label{sTime}
Assume $\l \in C^2[0,T)$ with $  \l'(0)=1/2$ and $\l(0)=0$, and let $ c>0$.  
For $t \in[0,T]$, set $s = s(t)= (\l_c')^{-1}(\l'(t))$.  
\begin{enumerate}
  \item If $ 0 < LC_ \g(t) \leq c^2/2$, then, $s'(t) \geq 1$.
  \item  If $ c^2/2 \leq LC_ \g(t) < \infty $, then, $s'(t) \leq 1$.
  \item  If $-\infty < LC_{\g}(t)  \leq -c^2/2, $ then $s'(t) \geq 1$.
  \item If $-c^2/2 \leq LC_{\g}(t) <0$, then $s'(t) \leq 1$.

\end{enumerate}
\end{lemma}

\begin{proof}
We assume $ 0 < LC_ \g(t) \leq c^2/2$ to prove the first statement.  The other statements are proved in the same manner.

The reparametrization function $s(t)$ is well-defined because $\l_c'(t)$ is strictly increasing from  $1/2$ to $\infty$ and $\l$ is  strictly increasing with $\l'(0)=1/2$.
Since $\l'_c(s) = \l'(t)$ and
  $$ \frac{\l'(t)^3}{\l''(t)} \leq \frac{c^2}{2} = \frac{\l_c'(s)^3}{\l_c''(s)},$$
 it must be true that $\l_c''(s) \leq \l''(t)$.
  Thus 
  $\displaystyle s'(t) = \frac{\l''(t)}{\l_c''(s)} \geq 1.$
\end{proof}

Now for the proof of Theorem \ref{below}:

\begin{proof}

We assume that $0< LC_{\g}(t) \leq c^2/2,$ and we will prove the first statement; the remaining statements are proved in a similar manner.  

The first step is to prove that the base of $\g^c$ is to the left of the base of $\g$.
Since $0< LC_{\g}(t) \leq c^2/2$,
then $$\frac{\l'(0)^3}{\l''(0)} \leq \frac{c^2}{2} = \frac{\l_c'(0)^3}{\l_c''(0)}.$$
Thus $\l''(0) \geq \l_c''(0)$, since $\l'(0) = \l_c'(0) = 1/2$.
Proposition \ref{SeriesProp} implies that
$$\g(t) =  2i\sqrt{t} + \frac{1}{2} \, t - i \frac{1}{32}\, t^{3/2} + \left( \frac{4}{15} \l''(0)+\frac{1}{1080} \right)\, t^2 + o( t^2 )$$
for $t$ near 0.
Fix $h>0$, and let $t$ and $t_c$ satisfy
$  \text{Im}\, \g(t) = h =\text{Im}\, \g^c(t_c) $. 
That is,
$$  2 \sqrt{t} - \frac{1}{32} t^{3/2} +o(t^2) = h = 2 \sqrt{t_c} - \frac{1}{32} t_c^{3/2} +o(t_c^2),$$
which implies that
$$t-t_c = o\left( (t \vee t_c)^2 \right).$$
Then
\begin{align*}
\text{Re}\, \g(t) - \text{Re}\, \g^c(t_c) 
	&= \left[ \text{Re}\, \g(t_c) - \text{Re}\, \g^c(t_c) \right] 
	     - \left[ \text{Re}\, \g(t_c) - \text{Re}\, \g(t) \right]  \\
	&= \frac{4}{15} \left( \l''(0) - \l_c''(0) \right) t_c^2 + o\left( (t \vee t_c)^2 \right)     
\end{align*}
So for $h$ small enough, $\text{Re}\, \g(t) - \text{Re}\, \g^c(t_c) >0$, proving that
 the base of the curve $\g^c$ is to the left of the base of the curve $\g$. 

Since $\g^c$ is self-similar, 
 $\g^c_t$ is simply a scaled version of $\g^c$, that is $\g^c_t = r \cdot \g^c$.
By Proposition \ref{SeriesProp}, the scale of  $\g^c_t$ is $2\l_c'(t)/3$,
and the scale of $r \cdot \g^c $  is $ (2 \l_c'(0)/3) \cdot (1/r) $.  
Thus $r = \l_c'(0) / \l_c'(t)$, and
 $$\g_t^c = \frac{1}{2\l_c'(t)} \cdot \g^c.$$
Since $\l_c'(t)$ is strictly increasing, $\g^c_t$ is ``shrinking" as time increases.

Suppose that $\g$ is not below $\g_c$.  
Then $\g$ and $\g^c$ must cross, and  
there exists $z_0 = \g(\tau)$, where $0<\tau < T$, so that $z_0$ is ``outside" $\g^c$.  
If $c \geq 4$,
$z_0$ is outside $\g^c$ when $z_0$ is in  
 the unbounded component of $\H \setminus \g^c$.
 If $c < 4$, 
 $z_0$ is outside $\g^c$, if there is a continuous curve $\beta$ joining $\g^c(\tau_c)$ to $\R\setminus \{0\}$ in $\H \setminus \left( \g^c \cup \g(0, \tau) \right)$ with $z_0$ in the unbounded component of $ \H \setminus \left( \g^c \cup \beta \right)$.
 Set $z_t = g_t(z_0) - \l(t)$ and set $w_t = g^c_{s(t)}(w_0)-\l_c(s(t))$ for a point $w_0 \in \H$ which will be specified later. 
Then
\begin{equation}\label{flows}
\partial_t \, z_t = \frac{2}{z_t} - \l'(t)  \;\; \text{  and  } \;\;
\partial_t \, w_t = \left( \frac{2}{w_t} - \l'(t) \right) \cdot s'(t)
\end{equation}
using the fact that $\l_c'(s(t)) = \l'(t)$.
Since $z_0 = \g(\tau)$, then $z_t \in \g_t$ for $0<t < \tau$.

We claim that $z_t$ is outside $\g^c_{s(t)}$ for all $t < T$.  
Let $t_0$ be a time that $z_{t_0}$ is outside of  $\g^c_{s({t_0})}$, and  choose $w_0 \in \H$ so that  $z_{t_0} = w_{t_0} $. Now $w_t \notin \g^c_{s({t})}$ for $t > t_0$.
By \eqref{flows} and Lemma \ref{sTime}, the direction of motion for $z_t$ and $w_t$ at time $t=t_0$ is the same, but $|\partial_t w_{t_0}| \geq |\partial_t z_{t_0}|$. 
Because of the shape of the curve $\g^c$, in order for $z_{t}$ to move below $\g^c_{s(t)}$, 
$z_t$ must either move faster than $w_t$ or in a different direction than $w_t$ (or both.)

Thus $z_t$ remains outside $\g^c_{s(t)}$ for all $t <T$.  However, this leads to the contradiction.  The kill-time for $z_0$ is $\tau$ and $\tau <T$.  However, the base of $\g_t$ always remains below (i.e. not outside) $\g^c_{s(t)}$, and so $z_0$ cannot be killed before time $T$.

\end{proof}

To round out  this section, we mention the following result which is related to the infinite curvature case:

\begin{prop}
Assume $\g$ is generated by $\lambda \in C^1[0,T)$ with  $\l(0)=0$, and let $c>0$. 
\begin{enumerate}
  \item If $\l'(t) \geq c$, then $\g$ is below the curve generated by $ct$.
  \item If $\l'(t) \leq c$, then the curve generated by $ct$ is below $\g$.
\end{enumerate}
\end{prop}

This can be proved in the same manner as Theorem \ref{below}.
\\

\end{document}